\documentclass{amsart}     
\usepackage{amsrefs}

\title{Non-collapsing in mean-convex mean curvature flow}

\author{Ben Andrews}
\address{Mathematical Sciences Institute, Australian National University, ACT 0200 Australia; Mathematical Sciences Center, Tsinghua University, Beijing 100084, China; Morningside Center for Mathematics, Chinese Academy of Sciences, Beijing 100190, China, }
\email{Ben.Andrews@anu.edu.au}
\newtheorem{thm}{Theorem}   
\newtheorem{prop}[thm]{Proposition}
\newtheorem{lem}[thm]{Lemma}          
\theoremstyle{definition}
\newtheorem{defn}[thm]{Definition}    
\newtheorem*{rem}{Remarks}             
\newcommand\RR{\mathbb{R}}

\newcommand{\xH}{H_x}
\newcommand{\xh}{h^x}
\newcommand{\yH}{H_y}
\newcommand{\yh}{h^y}
\newcommand{\xnu}{\nu_x}
\newcommand{\ynu}{\nu_y}

\begin{document}

\begin{abstract}    
We provide a direct proof of a non-collapsing estimate for compact hypersurfaces with positive mean curvature moving under the mean curvature flow:  Precisely, if every point on the initial hypersurface admits an interior sphere with radius inversely proportional to the mean curvature at that point, then this remains true for all positive times in the interval of existence.
\end{abstract}

\maketitle


We follow \cite{SW} in defining a notion of 'non-collapsing' for embedded hypersurfaces as follows:  Recall that a hypersurface $M$ is called \emph{mean-convex} if the mean curvature $H$ of $M$ is positive everywhere.

\begin{defn}
A mean convex hypersurface $M$ bounding an open region $\Omega$ in $\mathbb{R}^{n+1}$ is $\delta$-non-collapsed (on the scale of the mean curvature) if for every $x\in M$ there there is an open ball $B$ of radius $\delta/H(x)$ contained in $\Omega$   with $x\in\partial\Omega$.
\end{defn}

It was proved in \cite{SW} that any compact mean-convex solution of the mean curvature flow is $\delta$-non-collapsed for some $\delta>0$.  Closely related statements are deduced by Brian White in \cite{WhiteMCF}.  In both of these works the result is derived only after a lengthy analysis of the properties of solutions of mean curvature flow.  The purpose of this paper is to provide a self-contained proof of such a non-collapsing result using only the maximum principle.

It is first necessary to reformulate the non-collapsing condition to allow the application of the maximum principle.  Given a hypersurface $M=X(\bar M)$, define a function $Z$ on 
$M\times M$ by
$$
Z(x,y)=\frac{H(x)}{2}\|X(y)-X(x)\|^{2}+\delta\left\langle X(y)-X(x),\nu(x)\right\rangle.
$$
Then we have the following characterization:

\begin{prop}
$M$ is $\delta$-non-collapsed if and only if $Z(x,y)\geq 0$ for all $x,y\in\bar M$.
\end{prop}

\begin{proof}
By convention we choose the unit normal $\nu$ to be outward-pointing, so that a ball in $\Omega$ of radius $\delta/H(x)$ with $X(x)$ as a boundary point must have centre at the point $p(x)=X(x)-\frac{\delta}{H(x)}\nu(x)$.  The statement that this ball is contained in $\Omega$ is equivalent to the statement that no points of $M$ are of distance less than 
$\delta/H(x)$ from $p$:  
$$
0\leq \|X(y)-p(x)\|^{2}-\left(\frac{\delta}{H(x)}\right)^{2} = 2H(x)Z(x,y)
$$
for all $x$ and $y$ in $\bar M$.
Since $H>0$ this is equivalent to the statement that $Z$ is non-negative everywhere.  The converse is clear.
\end{proof}

The main result of this paper is the following:

\begin{thm}
Let $\bar M^n$ be a compact manifold, and $X: \bar M^n\times[0,T)\to\RR^{n+1}$ a family of smooth embeddings evolving by mean curvature flow, with positive mean curvature.  If $M_{0}=X(\bar M,0)$ is $\delta$-non-collapsed for some $\delta>0$, then $M_{t}=X(\bar M,t)$ is $\delta$-non-collapsed for every $t\in[0,T)$.
\end{thm}

\begin{proof}
By the Proposition, the Theorem is equivalent to the statement that the function 
$Z:\ \bar M\times\bar M\times[0,T)\to \mathbb{R}$ defined by
$$
Z(x,y,t)=\frac{H(x,t)}{2}\|X(y,t)-X(x,t)\|^{2}+\delta\left\langle X(y,t)-X(x,t),\nu(x,t)\right\rangle
$$
is non-negative everywhere provided that it is non-negative on 
$\bar M\times\bar M\times\{0\}$.  We prove this using the maximum principle.  For convenience we denote by
 $\xH$ the mean curvature and $\xnu$ the outward unit normal at $(x,t)$, and we write $d=|X(y,t)-X(x,t)|$ and $w=\frac{X(y,t)-X(x,t)}{d}$, and  $\partial^x_i=\frac{\partial X}{\partial x^i}$.

We compute the first and second derivatives of $Z$, with respect to some choices of local normal coordinates $\{x^i\}$ near $x$ and $\{y^i\}$ near $y$.  
\begin{equation}\label{eq:Zy}
\frac{\partial Z}{\partial y^i} = d\xH\langle w,\partial^y_i\rangle + \delta\langle \partial^y_i,\xnu\rangle.
\end{equation}
From this we have the following:

\begin{lem}
$$
\xnu+\frac{d\xH}{\delta}w - \frac{1}{\delta}\frac{\partial Z}{\partial y^q}g_y^{qp}\partial^y_p = 
\ynu\sqrt{1+\frac{2\xH}{\delta^2}Z-\frac{1}{\delta^2}|\nabla_yZ|^2}
$$
\end{lem}

\begin{proof}
Equation \eqref{eq:Zy} gives for each $i$
$$
0 = \langle \partial^y_i,\xnu+\frac{d\xH}{\delta}w\rangle - \frac{1}{\delta}\frac{\partial Z}{\partial y^i} = 
\langle \partial^y_i,\xnu+\frac{d\xH}{\delta}w-\frac{1}{\delta}\nabla_yZ\rangle,
$$
where $\nabla_yZ = \frac{\partial Z}{\partial y^k}g_y^{kl}\partial^y_l$.  Thus the vector 
$\xnu+\frac{d\xH}{\delta}w-\frac{1}{\delta}\nabla_yZ$ is normal to the hypersurface at $y$, and is a multiple of $\ynu$.  To complete the Lemma we compute the length of this vector:
\begin{align*}
\|\xnu+\frac{d\xH}{\delta}w-\frac{1}{\delta}\nabla_yZ\|^2
&=1+\left(\frac{d\xH}{\delta}\right)^2 +2\frac{d\xH}{\delta}\langle \xnu,w\rangle+\frac{1}{\delta^2}|\nabla_yZ|^2\\
&\quad\null-\frac{2}{\delta}\langle \nabla_yZ,\xnu+\frac{d\xH}{\delta}w\rangle\\
&=1+\left(\frac{d\xH}{\delta}\right)^2+2\frac{\xH}{\delta^2}\left(Z-\frac{d^2H_x}{2}\right)+\frac{1}{\delta^2}|\nabla_yZ|^2\\
&\quad\null-\frac{2}{\delta}\langle\nabla_yZ,\xnu+\frac{d\xH}{\delta}w-\frac{1}{\delta}\nabla_yZ\rangle
-\frac{2}{\delta^2}|\nabla_yZ|^2\\
&=1+\frac{2\xH}{\delta^2}Z-\frac{1}{\delta^2}|\nabla_yZ|^2,
\end{align*}
where we used the fact that $\nabla_yZ$ is in the tangent space at $y$, hence orthogonal to 
$\xnu+\frac{d\xH}{\delta}w-\frac{1}{\delta}\nabla_yZ$.
\end{proof}

Similarly we have (writing $\xh$ for the second fundamental form at $(x,t)$)
\begin{equation}\label{eq:Zx}
\frac{\partial Z}{\partial x^i} = -d\xH\langle w,\partial^x_i\rangle+\frac{d^2}{2}\nabla_i\xH
+\delta d\xh_{iq}g_x^{qp}\langle w,\partial^x_p\rangle.
\end{equation}

Now the second derivatives:
\begin{align}
\frac{\partial^2Z}{\partial y^i\partial y^j}
&=\xH\left\langle \partial^y_i,\partial^y_j\right\rangle-d\xH\yh_{ij}\left\langle w,\ynu\right\rangle
-\delta\yh_{ij}\langle \ynu,\xnu\rangle.\label{eq:Zyy}\\
\frac{\partial^2Z}{\partial y^i\partial x^j}
&=-\xH\langle \partial^x_j,\partial^y_i\rangle +d\langle w,\partial^y_i\rangle\nabla_j\xH
+\delta \xh_{jq}g_x^{qp}\langle \partial^y_i,\partial^x_p\rangle
\label{eq:Zxy}\\
\frac{\partial^2Z}{\partial x^i\partial x^j} &=
\xH\langle \partial^x_j,\partial^x_i\rangle -d\langle w,\partial^x_i\rangle\nabla_j\xH
+d\xH\xh_{ij}\langle w,\xnu\rangle
-d\langle w,\partial^x_j\rangle \nabla_i\xH\notag\\
&\quad\null+\frac{d^2}{2}\nabla_j\nabla_i\xH+\delta d\nabla_j\xh_{iq}g_x^{qp}\langle w,\partial^x_p\rangle
-\delta\xh_{ij}-\delta d \xh_{iq}g_x^{qp}\xh_{pj}\langle w,\xnu\rangle.\label{eq:Zxx}
\end{align}

Finally we compute the time derivative:
\begin{align}\label{eq:Zt}
\frac{\partial Z}{\partial t} &=d\xH\langle w,-\yH\ynu+\xH\xnu\rangle + \frac{d^2}{2}\left(\Delta\xH+\xH|\xh|^2\right)\\
&\quad\null+\delta\langle -\yH\ynu+\xH\xnu,\xnu\rangle + \delta d\langle w,\nabla\xH\rangle.\notag
\end{align}

We compute at a point $(x,y)$ of $Z$, with $y\neq x$.  Choose local coordinates so that $\{\partial^x_i\}$ are orthonormal, $\{\partial^y_i\}$ are orthonormal, and $\partial^x_i = \partial^y_i$ for $i=1,\dots,n-1$.  Thus $\partial^x_n$ and $\partial^y_n$ are coplanar with $\xnu$ and $\ynu$.  
  
Now compute
\begin{align*}
\frac{\partial Z}{\partial t}& - \sum_{i,j=1}^n\left(g_x^{ij}\frac{\partial^2Z}{\partial x^i\partial x^j}+g_y^{ij}\frac{\partial^2Z}{\partial y_i\partial y^j}+2g_x^{ik}g_y^{jl}\langle\partial^x_k,\partial^y_l\rangle\frac{\partial^2Z}{\partial x^i\partial y^j}\right)\\
&=d\xH\langle w,-\yH\ynu+\xH\xnu\rangle + \frac{d^2}{2}\left(\Delta\xH+\xH|\xh|^2\right)+\delta\langle -\yH\ynu+\xH\xnu,\xnu\rangle\\
&\quad\null + \delta d\langle w,\nabla\xH\rangle-n\xH+d\xH\yH\langle w,\ynu\rangle + \delta \yH\langle\ynu,\xnu\rangle-n\xH-d\xH^2\langle w,\xnu\rangle\\
&\quad\null+2d\langle w,\nabla\xH\rangle-\frac{d^2}{2}\Delta\xH-\delta d\langle w,\nabla\xH\rangle+\delta\xH+\delta d\langle w,\xnu\rangle|\xh|^2+2(n-1)\xH\\
&\quad\null+2\langle \partial^x_n,\partial^y_n\rangle^2\xH
-2d\langle\partial^x_i,\partial^y_j\rangle\langle w,\partial^y_j\rangle\nabla_i\xH-2\delta\left(\xH-\xh_{nn}+\langle \partial^x_n,\partial^y_n\rangle^2\xh_{nn}\right)\\
&=Z|\xh|^2+2d\left\langle w,\partial^x_i-\langle\partial^x_i,\partial^y_k\right\rangle g_y^{kl}\partial^y_l\rangle g_x^{ij}\nabla_j\xH-2\left(\xH-\delta\xh_{nn}\right)\left(1-\langle \partial^x_n,\partial^y_n\rangle^2\right).
\end{align*}
The second term on the last line can be rewritten in terms of the first derivatives of $Z$ using Equation \eqref{eq:Zx}:  This gives $\nabla_j\xH = \frac{2}{d^2}\frac{\partial Z}{\partial x^j}+\frac{2}{d}\langle w,\xH\partial^x_j-\delta\xh_{jp}g_x^{pq}\partial^x_q\rangle$.  Also we observe that $\partial^x_n-\langle\partial^x_n,\partial^y_n\rangle\partial^y_n = \langle \partial^x_n,\ynu\rangle\ynu$.  Therefore at any critical point of $Z$ we have
\begin{align*}
&\frac{\partial Z}{\partial t} - \sum_{i,j=1}^n\left(g_x^{ij}\frac{\partial^2Z}{\partial x^i\partial x^j}+g_y^{ij}\frac{\partial^2Z}{\partial y_i\partial y^j}+2g_x^{ik}g_y^{jl}\langle\partial^x_k,\partial^y_l\rangle\frac{\partial^2Z}{\partial x^i\partial y^j}\right)-|\xh|^2Z\\
&\quad= 
2\left(\xH-\delta\xh_{nn}\right)\left(2\langle w,\ynu\rangle^2\langle \partial^x_n,\ynu\rangle^2-\frac{2\delta}{d\xH}\langle w,\ynu\rangle\langle\partial^x_n,\ynu\rangle\langle \partial^x_n,\partial^y_n\rangle\langle \partial^y_n,\xnu\rangle-\langle\partial^x_n,\ynu\rangle^2\right).
\end{align*}
To simplify this we use Equation \eqref{eq:Zy} to write $\langle \partial^y_n,\nu_x\rangle = -\frac{d\xH}{\delta}\langle w,\partial^y_n\rangle$.  The first two terms in the bracket then become
$$
2\langle w,\ynu\rangle\langle \partial^x_n,\ynu\rangle\left(\langle w,\ynu\rangle\langle \partial^x_n,\ynu\rangle+\langle \partial^x_n,\partial^y_n\rangle\langle w,\partial^y_n\rangle\right)=
2\langle w,\ynu\rangle\langle \partial^x_n,\ynu\rangle\langle w,\partial^x_n\rangle.
$$
To simplify this we apply Lemma 1 twice:  In the first factor (writing $\rho=\sqrt{1+\frac{2\xH}{\delta^2}Z}$)
$$
\langle w,\ynu\rangle = \frac{1}{\rho}\langle w,\xnu+\frac{d\xH}{\delta}w\rangle
=\frac{1}{\rho}\left(\frac{Z}{d\delta}-\frac{d\xH}{2\delta}+\frac{d\xH}{\delta}\right) =
\frac{1}{\rho}\frac{d\xH}{2\delta}\left(1+\frac{2Z}{d^2\xH}\right).
$$
For the third factor we apply a rotation $J$ in the plane spanned by $\xnu$ and $\ynu$ (note that $w$ is in this plane by Equation \eqref{eq:Zy}), taking $\xnu$ to $\partial^x_n$, $\ynu$ to $\partial^y_n$, and $w$ to a vector $Jw$ which is orthogonal to $w$.  This gives
$$
\langle w,\partial^x_n\rangle = \langle w, \rho\partial^y_n-\frac{d\xH}{\delta}Jw\rangle = -\rho\frac{\delta}{d\xH}\langle \partial^y_n,\xnu\rangle.
$$
The three terms together then give
$-\left(1+\frac{2Z}{d^2\xH}\right)\langle \partial^x_n,\ynu\rangle\langle \partial^y_n,\xnu\rangle
= \left(1+\frac{2Z}{d^2\xH}\right) \langle \partial^x_n,\ynu\rangle^2$.
Finally we have
\begin{align*}
&\frac{\partial Z}{\partial t} = \sum_{i,j=1}^n\left(g_x^{ij}\frac{\partial^2Z}{\partial x^i\partial x^j}+g_y^{ij}\frac{\partial^2Z}{\partial y_i\partial y^j}+2g_x^{ik}g_y^{jl}\langle\partial^x_k,\partial^y_l\rangle\frac{\partial^2Z}{\partial x^i\partial y^j}\right)\\
&\quad\quad\null+\left(|\xh|^2+\frac{4\xH\left(\xH-\delta\xh_{nn}\right)}{\delta^2}\langle w,\partial^y_n\rangle^2\right)Z.
\end{align*}
Since the coefficient of $Z$ is a smooth function which is bounded on $(M\times M)\setminus \{x=y\}$, the maximum principle implies that $Z$ remains non-negative if initially non-negative ($Z$ is zero on the diagonal $\{y=x\}$). 
\end{proof}

\begin{rem}

\begin{enumerate}
\item The computation is valid for curve-shortening flow of a convex curve;
\item  The estimate implies curvature pinching, i.e. $\xH g-\delta\xh\geq 0$;
\item We made no use of the sign assumption on $\delta$, so the result also holds for negative $\delta$.  This proves `exterior non-collapsing', i.e. the hypersurface remains outside the ball of radius $|\delta|/\xH$ which touches the tangent plane at $x$ on the exterior;
\item The latter implies lower curvature pinching, i.e. $\xH g +|\delta|\xh\geq 0$;
\item The same computation shows that $Z$ remains non-positive if initially non-positive.  This applies in the case where $M$ is convex, and proves that if $M$ is contained in the ball of radius $\delta/\xH$ which touches the tangent plane at $x$ for every $x$ at the initial time, then this remains true for positive times.  In this situation this implies curvature pinching, i.e. $\xH g-\delta\xh \leq 0$.
\item In the latter case the conclusion is much stronger than pointwise curvature pinching:  It shows that the inradius and circumradius are both comparable to the reciprocal of the mean curvature at every point, and consequently that the mean curvatures at \emph{different} points are comparable.  The curvature pinching then implies that principal curvatures at different points are also comparable.  This allows a very simple proof of convergence of convex hypersurfaces to spheres under mean curvature flow, recovering both Huisken's theorem \cite{Huisken} for $n\geq 2$ and Gage and Hamilton's theorem \cites{Gage,GH} for $n=1$.
\item If the assumption of positive mean curvature is dropped, the conclusion still holds if we replace the mean curvature $H$ by any positive solution $f$ of the equation $\frac{\partial f}{\partial t} = \Delta f + \|A\|^2 f$.  In particular, this applies to prove a non-collapsing result if the initial hypersurface is star-shaped (see \cite{Smoczyk}).
\end{enumerate}
\end{rem}

\end{document}